\documentclass[11pt,reqno]{amsart}
\usepackage{hyperref}
\usepackage{pb-diagram, graphics}
\usepackage[usenames,dvipsnames]{color}
\usepackage{setspace}
\usepackage{amscd,amsfonts}
\usepackage{amssymb}
\usepackage{amsmath}
\usepackage{breqn}
\usepackage{amsthm}
\theoremstyle{plain}
\newtheorem{thm}{Theorem}[subsection]

\newtheorem{lem}[thm]{Lemma}

\theoremstyle{definition}
\newtheorem{defn}[thm]{Definition}

\newcommand{\F}{\mathcal{F}}

\def\Z{\mathbb{Z}}
\def\z0{\mathbb{Z}_{\geq0}}
\def\N{\mathbb{N}}

\def\F{\mathbb{F}}

\def\Fn{\F[x_1,\ldots,x_n]}

\def\bu{\mathbf{u}}
\def\ba{\mathbf{a}}
\def\br{\mathbf{r}}

\begin{document}

\title{A Newton-Girard Formula for Monomial Symmetric Polynomials}
\author{Samuel Chamberlin}
\address{Department of Mathematics and Statistics\\
Park University\\
Parkville, MO 64152}
\email{samuel.chamberlin@park.edu}
\author{Azadeh Rafizadeh}
\address{Department of Physics and Mathematics\\
William Jewell College\\
Liberty, MO 64068}
\email{rafizadeha@william.jewell.edu}

\begin{abstract}
The Newton-Girard Formula allows one to write any elementary symmetric polynomial as a sum of products of power sum symmetric polynomials and elementary symmetric polynomials of lesser degree. It has numerous applications. We have generalized this identity by replacing the elementary symmetric polynomials with monomial symmetric polynomials.
\end{abstract}

\maketitle

\section{Introduction}

Newton's identity for elementary symmetric polynomials (also called the Newton-Girard identity) has been applied in many areas of mathematics. For some recent applications see \cite{aa,gal,ho,ros,wal}.

The elementary symmetric polynomials are monomial symmetric polynomials where the partition is of the form $(1,\ldots,1)$. 
It seems natural to further generalize these identities by replacing the elementary symmetric polynomials with monomial symmetric polynomials. The purpose of this work is to prove such a generalization. The structure of the proof follows that in \cite{mead}.

In Section 2, we state the necessary definitions and the Newton-Girard Formula. In Section 3, we state and prove our generalization. In Section 4, we give an application to Lie algebras.

{\em Acknowledgements: We would like to thank Emily Scott and Macy Tush for many helpful calculations. We would also like to thank Ole Warnaar for suggesting the notation in the theorem.}

\section{Preliminaries}

Define $\N$ to be the set of positive integers and $\z0$ to be the set of nonnegative integers. So $\z0=\N\cup\{0\}$. Let $\F$ be a field and, for $n\in\N$, let $\Fn$ be the polynomial ring in $n$ indeterminates over $\F$. The symmetric group $S_n$ acts on $\Fn$ by permuting the variables. Let $R$ be the invariant subring, which is the ring of symmetric polynomials in $n$ indeterminates over $\F$.

A \emph{partition} is a sequence of nonnegative integers in non-increasing order, which contains only finitely many zero terms. The non-zero terms of a partition are called the \emph{parts}. If two partitions have the same parts we say they are the same. Given a partition $\lambda$ we define the \emph{length}, $l(\lambda)$ to be the number of parts of $\lambda$. If $a_1,\ldots,a_k,r_1,\ldots,r_k\in\N$ and $a_1>a_2>\dots>a_k$ define $\left(a_1^{r_1},\ldots,a_k^{r_k}\right)$ to be the partition having $r_i$ parts equal to $a_i$ for each $i\in\{1,\ldots,k\}$. Note that in this case $l(\lambda)=\sum r_i$. Given $\bu:=\left(u_1,\ldots,u_k\right)\in\Z^k$ define $|\bu|=\sum_{i=1}^k u_i$.

\begin{defn}
	Let $\lambda$ be a partition with $l(\lambda)\leq n$. Adding zeros if $l(\lambda)<n$, write $\lambda=(b_1,\ldots,b_n)$ with $b_1\geq\dots\geq b_n\geq0$. Then define the \emph{monomial symmetric polynomial given by $\lambda$ in $n$ variables} by 											$$m_{\lambda}=\sum_{\substack{\left(c_1,\ldots,c_n\right)\in\z0^n\\\textnormal{a nontrivial permutation of }\left(b_1,\ldots,b_n\right)}}x_1^{c_1}\dots x_n^{c_n}\in R.$$
    If $l(\lambda)>n$ we define $m_\lambda:=0$.
\end{defn}

\begin{defn}
	Let $k\in\N$ be given.
	\begin{enumerate}
    	\item Define the \emph{power sum symmetric polynomial with degree 			$k$} by 
    	$$p_{k}:=m_{(k)}=\sum_{i=1}^nx_i^k.$$
    
    	\item Define the \emph{elementary symmetric polynomial of degree $k$} 			by 
   		$$e_{k}:=m_{\left(1^k\right)}=\sum_{1\leq i_1<i_2<\dots<i_k\leq 			n}x_{i_1}\dots x_{i_k}.$$
    \end{enumerate}
\end{defn}

\subsection{Newton-Girard Formula}

The following formula, originally proven by Newton, has subsequently been proven in numerous ways. Two very different proofs can be found in \cite{mac} and \cite{mead}.

\begin{thm}[Newton-Girard Formula]\label{NGthm}
	Let $k\in\N$, then 
	\begin{equation}\label{NGeq}
    	ke_{k}=\sum_{i=1}^k(-1)^{i-1}p_{i}e_{k-i}.
    \end{equation}
\end{thm}



\section{A Generalization to the Monomial Symmetric Polynomials}

The following theorem is a generalization of the Newton-Girard formula to the monomial symmetric polynomials.

\begin{thm}\label{mainthm}
	Let $\lambda:=\left(a_1^{r_1},\ldots,a_k^{r_k}\right)$ be a partition with $l(\lambda)\leq n$, $a_1>\cdots>a_k$ and let $\ba:=\left(a_1,\ldots,a_k\right)\in\N^k$, then
    \begin{equation}\label{maineq}
    	l(\lambda)m_{\lambda}=\sum_{\bu}(-1)^{|\bu |-1}\binom{|\bu|}				{u_1,\ldots u_k} p_{\ba\cdot\bu}m_{\left(a_1^{r_1-							u_1},\ldots,a_k^{r_k-u_k}\right)}
    \end{equation}
    where the sum is over all nonzero weak compositions $\bu:=\left(u_1,\ldots,u_k\right)\in\z0^k$ with $u_i\leq r_i$ for all $i\in\{1,\ldots,k\}$.
\end{thm}

The proof will follow the structure of that in \cite{mead}. Before proving this theorem we state a lemma which allows one to write a product of a power sum symmetric polynomial and a monomial symmetric polynomial as a linear combination of monomial symmetric polynomials. It is a special case of Proposition 2.4 in \cite {gal}.

\begin{lem} Given $a\in\N$ and a partition $\lambda:=\left(b_1^{s_1},\ldots,b_k^{s_k}\right)$
	\begin{align*}
    	p_am_\lambda&=m_{\left(a,b_1^{s_1},\ldots,b_k^{s_k}\right)}+\sum_{i=1}^k m_{\left(a+b_i,b_1^{s_1},\ldots,b_{i-1}^{s_{i-1}},b_i^{s_i-1},b_{i+1}^{s_{i+1}},\ldots,b_k^{s_k}\right)}+\sum_{j=1}^k\delta_{a,b_j}s_j m_{\left(b_1^{s_1},\ldots,b_{j-1}^{s_{j-1}},b_j^{s_j+1},b_{j+1}^{s_{j+1}},\ldots,b_k^{s_k}\right)}\\
    &+\sum_{p,q=1}^k\delta_{a+b_p,b_q}s_q m_{\left(a+b_i,b_1^{s_1},\ldots,b_{q-1}^{s_{q-1}},b_q^{s_q+1},b_{q+1}^{s_{q+1}},\ldots,b_{p-1}^{s_{p-1}},b_p^{s_p-1},b_{p+1}^{s_{p+1}},\ldots,b_k^{s_k}\right)}
    \end{align*}
\end{lem}

We now prove Theorem \ref{mainthm} using this Lemma.

\begin{proof}[Proof of Theorem \ref{mainthm}]
By the Lemma:
\begin{align*}\label{maineq}
    	&\sum_{\bu}(-1)^{|\bu|-1}\binom{|\bu|}{u_1,\ldots u_k}p_{\ba\cdot\bu}m_{\left(a_1^{r_1-u_1},\ldots,a_k^{r_k-u_k}\right)}\\
        &=\sum_{\bu\neq\br}(-1)^{|\bu|-1}\binom{|\bu|}{u_1,\ldots u_k}\Bigg(m_{\left(\ba\cdot\bu,a_1^{r_1-u_1},\ldots,a_k^{r_k-u_k}\right)}+\sum_{i=1}^k m_{\left(\ba\cdot\bu+a_i,a_1^{r_1-u_1},\ldots,a_i^{r_i-u_i-1},\ldots,a_k^{r_k-u_k}\right)}\\
        &\hskip.5in+\sum_{j=1}^k\delta_{\ba\cdot\bu,a_j}\left(r_j-u_j\right) m_{\left(a_1^{r_1-u_1},\ldots,a_j^{r_j-u_j+1},\ldots,a_k^{r_k-u_k}\right)}\\
      &\hskip.5in+\sum_{x,y=1}^k\delta_{\ba\cdot\bu+a_x,a_y}\left(r_y-u_y\right) m_{\left(a_1^{r_1-u_1},\ldots,a_y^{r_y-u_y+1},\ldots,a_x^{r_x-u_x-1},\ldots,a_k^{r_k-u_k}\right)}\Bigg)+(-1)^{|\br|-1}\binom{|\br|}{r_1,\ldots,r_k}p_{\ba\cdot\br}\\
        &=\sum_{\bu\neq\br}(-1)^{|\bu|-1}\binom{|\bu|}{u_1,\ldots u_k}m_{\left(\ba\cdot\bu,a_1^{r_1-u_1},\ldots,a_k^{r_k-u_k}\right)}\\
        &\hskip.5in+\sum_{\bu\neq\br}(-1)^{|\bu|-1}\binom{|\bu|}{u_1,\ldots u_k}\sum_{i=1}^k m_{\left(\ba\cdot\bu+a_i,a_1^{r_1-u_1},\ldots,a_i^{r_i-u_i-1},\ldots,a_k^{r_k-u_k}\right)}+(-1)^{|\br|-1}\binom{|\br|}{r_1,\ldots,r_k}p_{\ba\cdot\br}\\
        &\hskip.5in+\sum_{\bu\neq\br}(-1)^{|\bu|-1}\binom{|\bu|}{u_1,\ldots u_k}\sum_{j=1}^k\delta_{\ba\cdot\bu,a_j}\left(r_j-u_j\right) m_{\left(a_1^{r_1-u_1},\ldots,a_j^{r_j-u_j+1},\ldots,a_k^{r_k-u_k}\right)}\\
        &\hskip.5in+\sum_{\bu\neq\br}(-1)^{|\bu|-1}\binom{|\bu|}{u_1,\ldots u_k}\sum_{x,y=1}^k\delta_{\ba\cdot\bu+a_x,a_y}\left(r_y-u_y\right) m_{\left( a_1^{r_1-u_1},\ldots,a_y^{r_y-u_y+1},\ldots,a_x^{r_x-u_x-1},\ldots,a_k^{r_k-u_k}\right)}
	\end{align*}
    
    \begin{align*}
        &=\sum_{\bu\neq\br}(-1)^{|\bu|-1}\binom{|\bu|}{u_1,\ldots u_k}m_{\left(\ba\cdot\bu,a_1^{r_1-u_1},\ldots,a_k^{r_k-u_k}\right)}\\
        &\hskip.5in+\sum_{i=1}^k\sum_{\bu\neq\br}(-1)^{|\bu|-1}\binom{|\bu|}{u_1,\ldots u_k}m_{\left(\ba\cdot\bu+a_i,a_1^{r_1-u_1},\ldots,a_i^{r_i-u_i-1},\ldots,a_k^{r_k-u_k}\right)}+(-1)^{|\br|-1}\binom{|\br|}{r_1,\ldots,r_k}p_{\ba\cdot\br}\\
        &\hskip.5in+\sum_{j=1}^k\sum_{\bu\neq\br}(-1)^{|\bu|-1}\binom{|\bu|}{u_1,\ldots u_k}\delta_{\ba\cdot\bu,a_j}\left(r_j-u_j\right) m_{\left(a_1^{r_1-u_1},\ldots,a_j^{r_j-u_j+1},\ldots,a_k^{r_k-u_k}\right)}\\
        &\hskip.5in+\sum_{x,y=1}^k\sum_{\bu\neq\br}(-1)^{|\bu|-1}\binom{|\bu|}{u_1,\ldots u_k}\delta_{\ba\cdot\bu+a_x,a_y}\left(r_y-u_y\right) m_{\left( a_1^{r_1-u_1},\ldots,a_y^{r_y-u_y+1},\ldots,a_x^{r_x-u_x-1},\ldots,a_k^{r_k-u_k}\right)}\\
        &=\sum_{t=1}^km_\lambda+\sum_{\substack{\bu\neq\br\\|\bu|>1}}(-1)^{|\bu|-1}\binom{|\bu|}{u_1,\ldots u_k}m_{\left(\ba\cdot\bu,a_1^{r_1-u_1},\ldots,a_k^{r_k-u_k}\right)}+\sum_{j=1}^k\left(r_j-1\right)m_\lambda\\
        &\hskip.5in+\sum_{i=1}^k\sum_{\bu\neq\br}(-1)^{|\bu|-1}\binom{|\bu|}{u_1,\ldots u_k}m_{\left(\ba\cdot\bu+a_i,a_1^{r_1-u_1},\ldots,a_i^{r_i-u_i-1},\ldots,a_k^{r_k-u_k}\right)}+(-1)^{|\br|-1}\binom{|\br|}{r_1,\ldots,r_k}p_{\ba\cdot\br}
        \\
        &\hskip.5in+\sum_{x,y=1}^k\sum_{\bu\neq\br}(-1)^{|\bu|-1}\binom{|\bu|}{u_1,\ldots u_k}\delta_{\ba\cdot\bu+a_x,a_y}\left(r_y-u_y\right) m_{\left( a_1^{r_1-u_1},\ldots,a_y^{r_y-u_y+1},\ldots,a_x^{r_x-u_x-1},\ldots,a_k^{r_k-u_k}\right)}\\
        &=km_\lambda+\sum_{\substack{\bu\neq\br\\|\bu|>1}}(-1)^{|\bu|-1}\binom{|\bu|}{u_1,\ldots u_k}m_{\left(\ba\cdot\bu,a_1^{r_1-u_1},\ldots,a_k^{r_k-u_k}\right)}\\
        &\hskip.5in+\sum_{i=1}^k\sum_{|\bu|>1}(-1)^{|\bu|-2}\binom{|\bu|-1}{u_1,\ldots,u_i-1,\ldots,u_k}m_{\left(\ba\cdot\bu,a_1^{r_1-u_1},\ldots,a_k^{r_k-u_k}\right)}+(-1)^{|\br|-1}\binom{|\br|}{r_1,\ldots,r_k}p_{\ba\cdot\br}\\
        &\hskip.5in+l(\lambda)m_\lambda-km_\lambda
        \\
        &\hskip.5in+\sum_{x,y=1}^k\sum_{|\bu|>1}(-1)^{|\bu|-2}\binom{|\bu|-1}{u_1,\ldots,u_x-1,\ldots,u_k}\delta_{\ba\cdot\bu,a_y}\left(r_y-u_y+\delta_{x,y}\right)m_{\left(a_1^{r_1-u_1},\ldots,a_y^{r_y-u_y+1},\ldots,a_k^{r_k-u_k}\right)}\\
        &=l(\lambda)m_\lambda+\sum_{\substack{\bu\neq\br\\|\bu|>1}}(-1)^{|\bu|-1}\binom{|\bu|}{u_1,\ldots u_k}m_{\left(\ba\cdot\bu,a_1^{r_1-u_1},\ldots,a_k^{r_k-u_k}\right)}\\
        &\hskip.5in+\sum_{|\bu|>1}(-1)^{|\bu|-2}\binom{|\bu|}{u_1,\ldots,u_k}m_{\left(\ba\cdot\bu,a_1^{r_1-u_1},\ldots,a_k^{r_k-u_k}\right)}+(-1)^{|\br|-1}\binom{|\br|}{r_1,\ldots,r_k}p_{\ba\cdot\br}\\
        &=l(\lambda)m_\lambda+\sum_{\substack{\bu\neq\br\\|\bu|>1}}(-1)^{|\bu|-1}\binom{|\bu|}{u_1,\ldots u_k}m_{\left(\ba\cdot\bu,a_1^{r_1-u_1},\ldots,a_k^{r_k-u_k}\right)}+(-1)^{|\br|-1}\binom{|\br|}{r_1,\ldots,r_k}p_{\ba\cdot\br}\\
        &\hskip.5in+\sum_{\substack{\bu\neq\br\\ |\bu|>1}}(-1)^{|\bu|-2}\binom{|\bu|}{u_1,\ldots,u_k}m_{\left(\ba\cdot\bu,a_1^{r_1-u_1},\ldots,a_k^{r_k-u_k}\right)}+(-1)^{|\br|-2}\binom{|\br|}{r_1,\ldots,r_k}p_{\ba\cdot\br}\\
        &=l(\lambda)m_\lambda
	\end{align*}
\end{proof}

\section{An Application to Lie Algebras}

The current algebra for the Lie algebra $\mathfrak{sl}_2$ is the vector space $\mathfrak{sl}_2\otimes\mathbb{C}[t]$ with Lie bracket given by extending the bracket
$$[x\otimes f,y\otimes g]=[x,y]\otimes fg=(xy-yx)\otimes fg$$
bilinearly.

For a Lie algebra $\mathfrak{a}$ let $U(\mathfrak{a})$ represent its universal enveloping algebra. Then the abelian algebra $U(h\otimes\mathbb{C}[t])$ is isomorphic to the algebra of symmetric functions in $n$ variables via the map $\left(h\otimes t^j\right)\mapsto p_j$.

In 1978, Howard Garland, while investigating the representation theory of the loop algebras, defined the elements $\Lambda_r\in U(\mathfrak{h}\otimes\mathbb{C}[t])$, \cite{G}. These elements satisfy the relation
    \begin{equation}\label{Gar}
    	k\Lambda_{k}=\sum_{i=1}^k(-1)^{i-1}(h\otimes t^i)\Lambda_{k-i}.
    \end{equation}
Therefore, the Newton-Girard formula implies that under the isomorphism above $\Lambda_k\mapsto e_k$.

In 2013, the first author, while investigating the representation theory of the map algebras, defined the elements $p(\lambda)\in U(\mathfrak{h}\otimes\mathbb{C}[t])$, \cite{cha}. These elements satisfy the relation
    \begin{equation}\label{Cha}
    	l(\lambda)p(\lambda)=\sum_{\bu}(-1)^{|\bu |-1}\binom{|\bu|}			{u_1,\ldots u_k}(h\otimes t^{\ba\cdot\bu})p\left(a_1^{r_1-u_1},\ldots,a_k^{r_k-u_k}\right).
    \end{equation}
Therefore, by Theorem \ref{mainthm} under the isomorphism above $p(\lambda)\mapsto m(\lambda)$.

\end{document}